\documentclass[12pt]{amsart}
\usepackage{graphicx}
\usepackage{amsmath, enumerate}
\usepackage{amsfonts, amssymb, amsthm, xcolor, stmaryrd}
\usepackage{tikz, tikz-cd, caption, tabu, longtable, mathrsfs, xtab}
 \usepackage{dsfont, cite, todonotes}
\usepackage[toc,page]{appendix}
\usepackage{hyperref}

\numberwithin{equation}{section}
\theoremstyle{plain}

\newtheorem{thm}{Theorem}[section]
\newtheorem{lemma}[thm]{Lemma}
\newtheorem{prop}[thm]{Proposition}

\theoremstyle{definition}

\newtheorem{exmp}[thm]{Example}

\theoremstyle{remark}
\newtheorem{rmk}[thm]{Remark}

\usepackage{leftidx, mathtools}
\AtBeginDocument{%
   \def\MR#1{}
}

\newcommand{\af}{\mathfrak{a}}
\newcommand{\Qb}{\mathbb{Q}}
\newcommand{\Rb}{\mathbb{R}}
\newcommand{\Zb}{\mathbb{Z}}
\newcommand{\Nm}{{\mathrm{Nm}}}
\newcommand{\Hc}{{\mathcal{H}}}
\newcommand{\SL}{{\mathrm{SL}}}
\newcommand{\Oc}{\mathcal{O}}
\newcommand{\GL}{{\mathrm{GL}}}
\newcommand{\lp}{\left (}
\newcommand{\rp}{\right )}
\newcommand{\df}{\mathfrak{d}}
\newcommand{\smat}[4]{\left(\begin{smallmatrix}
                 #1 & #2\\
                 #3 & #4
\end{smallmatrix}\right)}
\newcommand{\pmat}[4]{\begin{pmatrix}
                 #1 & #2\\
                 #3 & #4
\end{pmatrix}}
\newcommand{\Dc}{{\mathcal{D}}}
\newcommand{\Cb}{\mathbb{C}}
\newcommand{\Pb}{\mathbb{P}}
\newcommand{\ef}{\mathfrak{e}}
\newcommand{\sh}{\mathrm{sh}}
\newcommand{\ch}{\mathrm{ch}}
\newcommand{\Nb}{\mathbb{N}}
\newcommand{\varep}{\varepsilon}
\newcommand{\Sc}{\mathcal{S}}
\newcommand{\ebf}{{\mathbf{e}}}
\newcommand{\sgn}{\mathrm{sgn}}
\newcommand{\Cl}{{\mathrm{Cl}}}

\newcommand{\llam}{{\lambda}}
\newcommand{\bw}{\bar{w}}
\newcommand{\wpp}{{w^\perp}}
\newcommand{\Xp}{X^\perp}
\newcommand{\Yp}{Y^\perp}

\newcommand{\half}{{\tfrac{1}{2}}}
\newcommand{\Ab}{\mathbb{A}}
\newcommand{\Ac}{{\mathcal{A}}}
\newcommand{\tPsi}{{\tilde{\Psi}}}
\newcommand{\tr}{\mathrm{tr}}
\newcommand{\zbar}{{\overline{z}}}

\newcommand{\Zhat}{{\widehat{\mathbb{Z}}}}
\newcommand{\transpose}[1]{{\prescript{t}{}{#1}}}

\newcommand{\tef}{{\tilde{\mathfrak{e}}}}

\newcommand{\tz}{{\tilde{z}}}
\newcommand{\Abf}{\mathbb{A}_\mathbf{f}}
\renewcommand{\Re}{\mathrm{Re}}
\renewcommand{\Im}{\mathrm{Im}}

\begin{document}
\title{
Petersson Norm of Cusp Forms Associated to Real Quadratic Fields
}
\author{Yingkun Li}
\address{Fachbereich Mathematik,
Technische Universit\"at Darmstadt, Schlossgartenstrasse 7, D--64289
Darmstadt, Germany}
\email{li@mathematik.tu-darmstadt.de}

\begin{abstract}
In this article, we compute the Petersson norm of a family of weight one cusp forms constructed by Hecke and express it in terms of the Rademacher symbol and the regulator of real quadratic field.
\end{abstract}
\maketitle

\section{Introduction}

For an integral ideal $\af$ in a real quadratic field $F = \Qb(\sqrt{D}) \subset \Rb$ with fundamental discriminant $D$ and a positive integer $\kappa$, we have an even integral lattice $L = L_{\af, \kappa} = (\af, \kappa \Nm/\Nm(\af))$. It is anisotropic over $\Qb$ and has signature $(1, 1)$ over $\Rb$.
In \cite{Hecke26}, Hecke associated to such lattice $L$ a holomorphic, vector-valued cusp form $\vartheta_L$ of weight one, whose Fourier expansion is given explicitly in terms of the elements in the dual lattice $L^\vee$ (see \S \ref{subsec:Hecke_Theta}).
In this note, we will calculate the Petersson norm of $\vartheta_L$, which is defined by
\begin{equation}
  \label{eq:PeterssonNorm}
  \left\| \vartheta_L \right\|^2_{\mathrm{Pet}} := \int_{\Gamma \backslash \Hc} v |\vartheta_L(\tau)|^2 d \mu(\tau),
\end{equation}
with $\Gamma := \SL_2(\Zb)$ and $d \mu(\tau) := \frac{dudv}{v^2}$ the invariant measure.
Let $\varepsilon_\kappa \in \Oc_F^\times$ be the totally positive unit defined in \S \ref{subsec:Hecke_Theta}. 
It turns out that the Petersson norm of $\vartheta_L$ is a rational multiple of $\log \varepsilon_\kappa$, and the rational multiple can be expressed in terms of Rademacher symbols $\Psi(\gamma)$ of hyperbolic elements $ \gamma \in \Gamma$ (see \S \ref{subsec:Rademacher}).

Our result is as follows.
\begin{thm}
  \label{thm:main}
Let $g_L, g_\kappa \in \GL^+_2(\Qb)$ and $\gamma_{D, \kappa} \in \Gamma$ be defined as in \eqref{eq:gL}, \eqref{eq:gkappa} and \eqref{eq:gamma} respectively.
 Denote $\gamma_0 := g_\kappa \gamma_{D, \kappa} g_\kappa^{-1}$ and $\gamma_1 := g_L \gamma_{D, \kappa} g_L^{-1}$.
Then $\gamma_j \in \Gamma$ and 
\begin{equation}
  \label{eq:main}
    \left\| \vartheta_L \right\|^2_{\mathrm{Pet}} = 
- \frac{1}{12} \lp \Psi(\gamma_0) + \Psi(\gamma_1) \rp
 \log \varepsilon_\kappa.
\end{equation}
\end{thm}

\begin{rmk}
  In \cite{CL16}, we constructed a harmonic Maass form $\tilde{\vartheta}_L$ that maps to $\vartheta_L$ under the differential operator $\xi := 2iv \overline{\partial_{\overline{\tau}}}$. One can express $\| \vartheta_L  \|^2_{\mathrm{Pet}}$ as a rational linear combination of the principal part Fourier coefficients of $\tilde{\vartheta}_L$, which are rational multiples of the fundamental unit of $F$. Therefore, $\|\vartheta_L\|^2_{\mathrm{Pet}} = c_L \log \varepsilon_\kappa$ for some $c_L \in \Qb$ by the result loc.~cit. 
\end{rmk}

In \cite{Hecke26}, Hecke gave a well-known example of $\vartheta_L$ by taking $D = 12, \kappa = 1$ and $\af = \df = \sqrt{D}\Oc_F$. The cusp form $\vartheta_L(\tau)$ has four nonzero components, all equal to $\eta^2(\tau)$ up to sign. {In this case, $\varepsilon_\kappa = 7 + 2\sqrt{D}$ and $\gamma_0 = \gamma_1 = \smat{7}{4}{12}{7}$.
From the definition of $\Psi$ in Eq.\ \eqref{eq:Psi}, we can evaluate $\Psi(\gamma_0) = \Psi(\gamma_1) = -2$. Therefore the theorem above implies that $\|\vartheta_L\|^2_{\mathrm{Pet}} = 4 \|\eta^2\|^2_{\mathrm{Pet}}  = \frac{2\log( 2 + \sqrt{3})}{3}$. 
On the other hand, if one takes $D = 12, \kappa = 1$ and $\af = \Oc_F$, then $\gamma_0$ is unchanged, but $\gamma_1 = \smat{7}{12}{4}{7}$ and $\Psi(\gamma_1) = 2$, which implies that $\vartheta_L(\tau)$ vanishes identically.} The same is true when $12$ is replaced with any positive, even fundamental discriminant.

The idea of the proof is to interchange the order of integration and evaluate the theta integral
$$
\Phi(t, t_0; L) := \int_{\Gamma \backslash \Hc} v \langle \Theta_L(\tau, t),  \Theta_L(\tau, t_0)  \rangle \frac{dudv}{v^2},
$$ 
where $\Theta_L$ is the integral kernel used to produce $\vartheta_L$ (see \S \ref{subsec:Hecke_Theta}) and $\langle, \rangle$ is the Hermitian pairing on finite dimensional complex vector spaces induced by the $L_2$-norm.
It turns out that the integral of $\Phi$ over $t$ is a constant independent of $t_0$. 
An important observation is that this integrand is the theta function attached to a unimodular lattice $M$ containing $L \oplus -L$, which has signature $(2, 2)$.
We can then evaluate $\Phi(t, t_0; L)$ using its Fourier expansion along a 1-dimensional boundary of an $\mathrm{O}(2, 2)$ Shimura variety. 
This step is a simpler version of the calculations in \cite{Kudla16} (see \cite{Gritsenko12} for a nice example), except the infinity component of the Schwartz function being slightly different.

Since $\vartheta_L$ is itself a theta lift, the Rallis inner product formula is a natural tool to evaluate its Petersson norm \cite{Ra87}. The most general case of this formula was recently completed by Gan, Qiu and Takeda \cite{GQT14}. In the notation loc.~cit., the case at hand is outside both the convergent range and the first term range. 
Gan, Qiu and Takeda extended the Rallis inner product formula to such cases by providing a missing ingredient, a second term identity of the regularized Siegel-Weil formula.
Our result can be seen as a more explicit version of Theorem 1.3 loc.~cit.\ for the dual pair $(\SL_2, \mathrm{O}(1, 1))$. The proof here does not use the regularized Siegel-Weil formula, and is more straightforward using Fourier expansions of theta lifts. Furthermore, the calculations here will be useful for other theta lifts involving real-analytic modular forms that we are currently pursuing \cite{Li18}.

The paper is organized as follows. We recall the basic setups in \S \ref{sec:setup}, then proceed to the necessary calculations in \S \ref{sec:calculations}, before proving Theorem \ref{thm:main} in \S \ref{sec:proof}.

\section*{Acknowledgment}
We thank \"{O}.\ Imamoglu for helpful conversations about cycle integrals and the Rademacher symbol, and the referee for helpful comments. 
The author is partially supported by the DFG grant BR-2163/4-2 and an NSF postdoctoral fellowship.

\section{Setup.}
\label{sec:setup}
\subsection{Quadratic Spaces and Embeddings.}
Let $(V, Q)$ be a rational quadratic space of signature $(p, q)$. 
The Grassmannian $\Dc$ of oriented, negative definite subspaces of dimension $q$ is a symmetric space. 
When $q = 2$, there is a complex structure on the Grassmannian by identifying it with
$$
\{Z \in V(\Cb): Q(Z) = 0, Q(\Im(Z)) < 0\} / \Cb^\times \subset \Pb(V(\Cb)).
$$
For example, take $(V, Q) = (M_2, \det/N)$, where $M_2(R)$ is the space of $2$ by $2$ matrices with entries in {an additive subgroup $R \subset \Cb$}, and $N \in \Qb_{> 0}$.
Let $\Dc^+$ be a connected component of $\Dc$ isomorphic to $\Hc^2$ via the map
\begin{align*}
\Hc^2   &\to \Dc^+ \\
(z_1, z_2) &\mapsto  \Cb Z(z_1, z_2),
\end{align*}
where $Z(z_1, z_2) :=  \smat{z_1z_2}{z_2}{z_1}{1} \in M_2(\Cb)$ for any $z_1, z_2 \in \Cb$.
Under this identification, the action of $\mathrm{SO}(2, 2)$ on $\Dc^+$ is the action of $\SL_2 \times \SL_2$ on $\Hc^2$.

Another example comes from number field.
Let $F = \Qb(\sqrt{D}) \subset \Rb$ be a real quadratic field with $'$ the Galois conjugation in $\mathrm{Gal}(F/\Qb)$.
The rational quadratic space $V_{F, N} := (F, \Nm/N)$ has signature $(1, 1)$. 
Since $F\otimes_\Qb \Rb \cong \Rb^2$ via the two real embeddings of $F$, the symmetric space $\Dc_{F, N}$ associated to $V_{F, N}$ is a subset of $\Rb^2$, whose connected component $\Dc_{F, N}^+$ is isomorphic to $\Rb$ via
\begin{equation}
  \label{eq:W}
  \begin{split}
      \Rb & \to \Dc^+_{F, N} \subset \Rb^2 \\
t &\mapsto Z(t) := \sqrt{N} \binom{e^t}{-e^{-t}}.
  \end{split}
\end{equation}
The orthogonal complement of $Z(t)$ is spanned by the vector $Z^\perp(t) := \sqrt{N} \binom{e^t}{e^{-t}}$.

With the quadratic form $Q = \det/N$, the following vector space is also a rational quadratic space of signature $(2, 2)$
\begin{equation}
  \label{eq:V}
V := \left\{\pmat{\lambda_1}{\lambda_2}{\lambda_2'}{\lambda_1'} \in M_2(F) \right\}.
\end{equation}
We can embed $V_{F, N}$ and $-V_{F, N}$ into $V$ via the the maps $\lambda \mapsto \smat{\lambda}{}{}{\lambda}$ and $\lambda \mapsto \smat{}{\lambda}{\lambda}{}$.
This induces an isometry $V \cong V_{F, N} \oplus -V_{F, N}$ of rational quadratic spaces and an embedding of $(\Dc_{F, N}^+)^2 \cong \Rb^2$ into the symmetric space $\Dc_V$ associated to $V$ via
$$
w = w(t, t_0) = X(t) + iY(t_0) :=  {\sqrt{N}}{} \smat{e^t}{ie^{t_0}}{ie^{-t_0}}{-e^{-t}} \in V(\Cb),
$$
which satisfies $(w, w) = 0$ and $(w, \bw) = -1 < 0$ for any $(t, t_0) \in \Rb^2$.
Meanwhile, we also define $\wpp(t, t_0)  \in V(\Cb)$ by 
$$
\wpp = \wpp(t, t_0) = \Xp(t) + i \Yp(t_0) := {\sqrt{N}}{} \smat{e^t}{ie^{t_0}}{-ie^{-t_0}}{e^{-t}} \in V(\Cb).
$$
Then $\{X(t), Y(t_0), \Xp(t), \Yp(t_0)\}$ form an orthogonal basis of $V(\Rb)$. 
For $\llam \in V(\Rb)$, we denote $\llam_w$ and $\llam_\wpp$ the projections of $\llam$ onto the plane spanned by $\{X(t), Y(t_0) \}$ and $\{\Xp(t), \Yp(t_0) \}$ respectively.
It is then easy to see that 
\begin{equation}
  \label{eq:norm_proj}
  Q(\llam_w) = - \frac{  (\llam, X(t))^2 +  (\llam, Y(t_0))^2}{4},   Q(\llam_\wpp) = \frac{ (\llam, \Xp(t))^2 + (\llam, \Yp(t_0))^2}{4},
\end{equation}
where $(,)$ is the bilinear form associated to $Q$. 
We define $Q_w(\llam) := Q(\llam_\wpp) - Q(\llam_w)$ to be the majorant associate to $w$.

\subsection{Witt Decomposition.}
\label{subsec:Witt}
The quadratic space $V$ has a natural Witt decomposition $V = U + U^\vee$, where $U, U^\vee$ are isotropic $\Qb$-subspaces of $V$ defined by
\begin{equation}
  \label{eq:U}
  U := \left\{\pmat{\mu}{-\mu}{-\mu'}{\mu'} \in M_2(F) \right\}  ,
  U^\vee := \left\{\pmat{\mu}{\mu}{\mu'}{\mu'} \in M_2(F) \right\}  .
\end{equation}
We fix a basis $\{\ef_1, \ef_2 \}$ of $U$ with corresponding dual basis $\{\ef^\vee_1, \ef^\vee_2\}$ of $U^\vee$ by
\begin{equation}
  \label{eq:ebasis}
  \ef_1 := \frac{\sqrt{D}}{2} \pmat{1}{-1}{1}{-1}, 
  \ef_2 := \frac{{1}}{2} \pmat{1}{-1}{-1}{1},
  \ef^\vee_1 := \frac{N}{2\sqrt{D}} \pmat{-1}{-1}{1}{1}, 
  \ef^\vee_2 := \frac{N}{2} \pmat{1}{1}{1}{1}.
\end{equation}
For $\llam \in V$, write $\llam_U$ and $\llam_{U^\vee}$ the $U$ and $U^\vee$ component of $\llam$ respectively. 
With respect to the basis \eqref{eq:ebasis}, we can view $\llam_U$ and $\llam_{U^\vee}$ as column vectors in $\Qb^2$ and write
$$
Q(\llam) = {\prescript{t}{}{\llam_U} \cdot \llam_{U^\vee}}
$$
Therefore, $V$ is isomorphic to $M_2(\Qb)$ as rational quadratic spaces via 
\begin{equation}
  \label{eq:isom}
  \begin{split}
\iota:  V &\stackrel{\cong}{\to}  M_2(\Qb) \\
\llam &\mapsto \pmat{(\llam, \ef^\vee_2)}{-(\llam, \ef^\vee_1)}{(\llam, \ef_1)}{(\llam, \ef_2)}.
\end{split}
\end{equation}
This isometry allows us to identify $\Dc^+_V$ with $\Dc^+ \cong \Hc^2$ by sending $w \in \Dc^+_V$ to $(w, \ef_2)^{-1} \iota(w) = Z(z_1, z_2) \in M_2(\Cb)$, where
\begin{equation}
  \label{eq:zj}
  z_1:= \frac{(w, \ef_1)}{(w, \ef_2)}, \quad
    z_2:= - \frac{(w, \ef^\vee_1)}{(w, \ef_2)}.
\end{equation}
Suppose $\{\tef_1, \tef_2\} = \{a \ef_1 + b \ef_2, c \ef_1 + d \ef_2\}$ is a different rational basis of $U$ with $g := \smat{a}{b}{c}{d} \in \GL^+_2(\Qb)$ and dual basis $\{\tef^\vee_1, \tef^\vee_2\} = \{ \tfrac{d \ef^\vee_1 - c \ef^\vee_2}{\det{g}}, \tfrac{-b \ef^\vee_1 + a \ef^\vee_2}{\det{g}} \}$. 
Then the corresponding point $(\tz_1, \tz_2) \in \Hc^2$ defined as in \eqref{eq:zj} satisfies
\begin{equation}
  \label{eq:tz}
  \tz_1 = g \cdot z_1, \quad \tz_2 = \frac{z_2}{\det g}.
\end{equation}
Let $\iota_g: V \to M_2(\Qb)$ be the isomorphism with respect to this basis.

With respect to the basis \eqref{eq:ebasis}, $\Cb \cdot \iota(w)$ is a point on $\Dc$ where
\begin{equation}
  \label{eq:iotaw}
  \begin{split}
 \iota(w(t, t_0)) &= \frac{ \sh(t) + i \ch(t_0)}{\sqrt{N}} Z(z_1(t, t_0), z_2(t, t_0)), \\
 z_1(t, t_0) &= \sqrt{D} \frac{- \sh (2 t) - \sh(2 t_0) + 2 i \ch(t-t_0) }{2(1 + \sh^2 t + \sh^2 t_0)}, z_2(t, t_0) = \frac{N}{{D}} z_1(t, -t_0).
 \end{split}
\end{equation}
Furthermore, $\iota(\wpp) =  \frac{(\ch(t) + i \sh(t_0))}{\sqrt{N}} \smat{\overline{z_1} z_2}{z_2}{\overline{z_1}}{1}$.
For $t \in \Rb$, we define
\begin{equation}
  \label{eq:z}
  z(t) = x(t) + i y(t) := \sqrt{D} \frac{-\sh(t) + i}{\ch(t)}.
\end{equation}
Note that $|z(t)| = \sqrt{D}$ for all $t \in \Rb$.
When $t_0 = 0$, we have $z_1(t, 0) = D z_2(t, 0) /N = z(t)$.
For any $g = \smat{a}{b}{c}{d} \in \GL_2^+(\Qb)$, the vectors $w(t, 0)$ and $\wpp(t, 0)$ become
\begin{equation}
  \label{eq:iotaw0}
\begin{split}
  \iota_g(w(t, 0)) &= \frac{-\overline{z(t)}}{\sqrt{N} y(t)} (cz(t) + d) Z(\tz_1(t), \tz_2(t)), \\
  \iota_g(\wpp(t, 0)) &= \frac{\sqrt{D}}{\sqrt{N} y(t)} (c\overline{z(t)} + d) Z(\overline{\tz_1(t)}, \tz_2(t))
\end{split}
\end{equation}
with $\tz_1(t) = g z(t), \tz_2(t) = \frac{N z(t)}{D \det{g}}$.
It is straightforward to check that 
\begin{equation}
  \label{eq:differential}
  d\tz_1 = \frac{i\det(g) yz}{(cz + d)^2 \sqrt{D}} dt, \;
d\tz_2 = \frac{i N yz}{\det(g) D^{3/2}} dt.
\end{equation}

\subsection{Hecke's Theta Function.}
\label{subsec:Hecke_Theta}
Let $N \in \Qb_{> 0}$ be the same as in the previous two sections.
For an $\Oc_F$-ideal $\af$ such that 
\begin{equation}
\label{eq:kappa}
\kappa := \Nm(\af) /N \in \Nb,
\end{equation}
the lattice $L = L_{\af, \kappa} := (\af, \Nm/N) \subset V_{F, N}$ is even integral, and its dual $L^\vee$ is given by the fractional ideal $(\kappa \df)^{-1} \af$.
Let $\Gamma_{\kappa}^+ \cong \Zb$ be subgroup of the discriminant kernel of this lattice that also fixes $\Dc^+_{F, N}$. It is generated by the smallest totally positive unit $\varep_{\kappa} \in \Oc_F^{\times}$ satisfying $\varep_\kappa > 1$ and $(\varep_\kappa - 1)\af^\vee \subset \af$. The last condition is equivalent to 
\begin{equation}
\label{eq:varep_cond}
\varep_\kappa - 1 \in \kappa \df.
\end{equation}
Let $\Sc_{L}$ be the space of $\Cb$-valued Schwartz-Bruhat functions on $L \otimes \Abf$ with support on $L^\vee \otimes \Zhat$ and that are translation-invariant under $L \otimes \Zhat$.
It is a $\Cb$-vector space of dimension $|L^\vee/L| = D \kappa$. 
There is the usual Weil representation $\omega$ of $\SL_2(\Abf)$ on $\Sc(V(\Abf))$ with respect to the standard additive character $\psi_\mathrm{f}$ of $\Abf$. It becomes the Weil representation $\rho_L$ in \cite{Borcherds98} when restricted to the subspace $\Sc_L$ \cite{Kudla03}.
To define Hecke's theta function, we begin with a theta kernel with value in $\Sc_{L}$
\begin{equation}
  \label{eq:Thetaaf}
  \Theta_{L}(\tau, t) := \sqrt{v} \sum_{\lambda \in L^\vee} 
(\lambda, Z^\perp(t))_{F, N} \ebf \lp \frac{(\lambda, Z^\perp(t))^2_{F, N}}{4} \tau - \frac{(\lambda, Z(t))^2_{F, N}}{4} \overline{\tau} \rp \varphi_\lambda
,
\end{equation}
where $(\cdot, \cdot)_{F, N}$ is the bilinear form associated to $\Nm/N$ and $\varphi_\lambda$ is the characteristic function of $\lambda + L \otimes \Zhat \subset V_{F, N}(\Abf)$.
It is $\Gamma^+_\kappa$ invariant as a function of $t \in \Dc_{F, N}^+$, and transforms in $\tau \in \Hc$ with respect to the Weil representation $\omega$ of weight 1 on $\Gamma$.
Integrating $\Theta_{L}(\tau, t) {dt}{}$ then defines a weight one modular form $\vartheta_{L}(\tau)$ with the following Fourier expansion (see \cite{Hecke26, CL16})
\begin{equation}
  \label{eq:vartheta}
  \vartheta_{L}(\tau) := \int_{\Gamma^+_\kappa \backslash \Dc_{F, N}^+} \Theta_{L}(\tau, t) {dt}{} 
=
\sum_{\begin{subarray}{c} \lambda \in \Gamma^+_\kappa \backslash L^\vee \\ \Nm(\lambda) > 0 \end{subarray}} \sgn(\lambda) \ebf \lp \frac{\Nm(\lambda)}{N} \rp
\varphi_\lambda.
\end{equation}
From this, it is clear that $\vartheta_{L} \in S_1(\Gamma, \rho_L) \otimes \Sc_{L}$, and that it only depends on the class of $\af$ in the narrow class group $\Cl^+(F)$ and the integer $\kappa$.
Finally, we can identify $\Sc_L$ with $\Cb^{D\kappa}$ by sending $\varphi \in \Sc_L$ to $(\varphi(\mu))_{\mu \in L^\vee/L}$, and let $\langle \cdot, \cdot \rangle$ be the standard Hermitian pairing induced by the $L^2$-norm $| \cdot |$. 
Then $v \left| \vartheta_L(\tau) \right|^2$ is an integrable, $\Gamma$-invariant function on $\Hc$ and the Petersson norm of $\vartheta_L(\tau)$ is defined as in Eq.\ \eqref{eq:PeterssonNorm}.

\subsection{Lattice in $V$.}
\label{subsec:lattice}
From the lattice $L = L_{\af, \kappa}$ in the previous section, we can construct the following unimodular lattice in $V$ 
$$
M = M_{\af, \kappa} := \left\{ \smat{\lambda_1}{\lambda_2}{\lambda_2'}{\lambda_1'} \in M_2(L^\vee): 
\lambda_1 - \lambda_2 \in L
\right\}.
$$
Let $\Gamma_M$ be its discriminant kernel and $\Gamma_M^+$ the subgroup fixing $\Dc^+$.
Then $\Gamma_M^+$ contains $(\Gamma_\kappa^+)^2$.
As in section 1.3 of \cite{Kudla16}, we can define a sublattice $P = P_U + P_{U^\vee} \subset M$ by
\begin{equation}
  \label{eq:P}
  \begin{split}
      P_U &:= M \cap U = \left\{ \pmat{\mu}{-\mu}{-\mu'}{\mu'} : \mu \in L^\vee \cap \frac{1}{2} L \right\}, \\
P_{U^\vee} &:= M \cap U^\vee = \left\{ \pmat{\mu}{\mu}{\mu'}{\mu'} : \mu \in L^\vee \right\}.
  \end{split}
\end{equation}
Via the map $\smat{\lambda_1}{\lambda_2}{\lambda'_1}{\lambda'_2} \mapsto \lambda_1 - \lambda_2$, the group $M/P$ is isomorphic to $L/(2L^\vee \cap L)$, which is then isomorphic to $(\Zb/2 \Zb)^2$ if $2 \nmid D\kappa$ and trivial otherwise.
In the former case, $P^\vee = \frac{1}{2} P$.

Suppose $L \cap 2 L^\vee$ has the $\Zb$-basis $\{a \sqrt{D} + b, d\}$ with $a, b, d \in \Qb$ and $a , d > 0$. Then $\{\tef_1, \tef_2\} = \{a \ef_1 + b \ef_2, d \ef_2\}$ is a $\Zb$-basis of $P_U$. 
We will denote 
\begin{equation}
\label{eq:gL}
g_L := \pmat{a}{b}{0}{d} \in \GL_2^+(\Qb).
\end{equation}
Note that
\begin{equation}
  \label{eq:detgL}
2  \det(g_L) = \Nm(\af \cap 2 (\kappa \df)^{-1}\af).
\end{equation}
In addition, we also define $g_\kappa \in \GL^+_2(\Qb)$ by
\begin{equation}
  \label{eq:gkappa}
  g_\kappa := 
  \begin{cases}
    \smat{2/(D\kappa)}{}{}{1} & 2 \mid D\kappa,\\
\smat{1/(D \kappa)}{1}{}{2} & 2 \nmid D \kappa.
  \end{cases}
\end{equation}
Then $g_\kappa \cdot \binom{\sqrt{D}/2}{1/2}$ is a $\Zb$-basis of the inverse different $\df_\kappa^{-1}$ of the order $\kappa \Zb + \frac{\kappa^2 D + \kappa\sqrt{D}}{2} \Zb$.

For an element $\lambda \in V$, suppose $\lambda_U = r, \lambda_{U^\vee} = \eta_1$ with respect to the basis $\{\tef_1, \tef_2, \tef^\vee_1, \tef^\vee_2\}$, then $\lambda_{U^\vee} \cdot P_U \subset \Zb$ if and only if $\eta_1 \in \Zb^2$. 
We use $\varphi_\lambda \in \Sc(V(\Abf))$ to denote the characteristic function of the coset $\lambda + P \otimes \Zhat$.

\begin{exmp}
\label{example:1}
Suppose $2 \nmid D \kappa$ and $\af = \Zb \frac{\sqrt{D} + B}{2} + \Zb A$ with $A$ odd, i.e.\ $a = 1, b = B, d = 2A$.
Then $\det(g) = 2A = 2 \Nm(\af)$ and $\kappa = A/N$.
Any $\lambda \in P^\vee/P$ can be written as $\lambda_U + \lambda_{U^\vee}$ with $\lambda_U \in \half P_U / P_U, \lambda_{U^\vee} \in \half P_{U^\vee}/P_{U^\vee}$.
Suppose $\lambda \in M/P$, then $\lambda_U \in P_U$ if and only if $\lambda_{U^\vee} \in P_{U^\vee}$.
Therefore the map $\lambda \mapsto \lambda_U$ (resp.\ $\lambda \mapsto P_{U^\vee}$) induces an isomorphism from $M/P$ to $\half P_U/P_U$ (resp.\ $\half P_{U^\vee}/P_{U^\vee}$).
For $\mu \in \half P_U/P_U$, let $\mu^\vee \in \half P_{U^\vee}/P_{U^\vee}$ be its image under the isomorphism above.
Let $\varphi_M \in \Sc(V(\Abf))$ be the characteristic function of $M$. Then it can be rewritten as
\begin{equation}
\label{eq:varphiM1}
\varphi_M(x) = \sum_{\lambda \in M/P} \varphi_\lambda(x) = 
\sum_{\mu \in \half P_U/P_U} \varphi_{\mu}(x_U) \varphi_{\mu^\vee}(x_{U^\vee}).
\end{equation}
The basis $\{\tef_1, \tef_2, \tef_1^\vee, \tef_2^\vee\}$ is given by
$$
\tef_1 = \smat{\frac{\sqrt{D} + B}{2}}{- \frac{\sqrt{D} + B}{2}}{\frac{\sqrt{D} - B}{2}}{-\frac{\sqrt{D} + B}{2}}, \;
\tef_2 = A \smat{1}{-1}{-1}{1}, \;
\tef_1^\vee = \frac{A\sqrt{D}}{2D \kappa} \smat{-1}{-1}{1}{1}, \;
\tef_2 = \smat{\frac{B\sqrt{D} + D}{4D \kappa}}{\frac{B\sqrt{D} + D}{4D \kappa}}{\frac{-B\sqrt{D} + D}{4D \kappa}}{\frac{-B\sqrt{D} + D}{4D \kappa}}.
$$
We can identify $\half P_U / P_U$, resp.\ $\half P_{U^\vee} / P_{U^\vee}$, with $(\half\Zb/\Zb)^2$, resp.\ $(\Zb/2\Zb)^2$, via the bases $\{\tef_1, \tef_2\}$, resp.\ $\{\tef_1^\vee, \tef_2^\vee\}$.
Under this identification, we have $\mu^\vee = \binom{0}{0}, \binom{1}{0}, \binom{0}{1}, \binom{1}{1}$ when $\mu = \binom{0}{0}, \binom{0}{1/2}, \binom{1/2}{0}, \binom{1/2}{1/2}$ respectively.
\end{exmp}

\subsection{Weil Representation and Theta Distribution.}
We quickly recall the Weil representation following \S 4.1 in \cite{Kudla16}.
Recall that $\Gamma = \SL_2(\Zb)$ and $\omega$ is the Weil representation of $\SL_2(\Ab)$ acting on the space of Schwartz functions $\Sc(V(\Ab))$, which comes from the usual polarization $V \otimes W = V\otimes X + V \otimes Y$ for $W = X + Y$ a 2-dimensional symplectic vector space.
On the other hand, the Witt decomposition $V = U + U^\vee$ gives another polarization $V \otimes W = U \otimes W + U^\vee \otimes W$, which provides another model of the Weil representation on $\Sc(U^\vee \otimes W(\Ab))$.
For a choice of basis, we can identify $U^\vee(\Ab)$ and $U(\Ab)$ with $\Ab^2$ and view $\varphi  \in \Sc(V(\Ab))$ as a function in $( \eta_1, r)$ with $\eta_1, r \in \Ab^2$.
Let $\psi$ be the standard additive character on $\Ab/\Qb$ that is trivial on $\Zhat$ and restricts to $x \mapsto \ebf(x)$ on $\Rb$.
For $\varphi \in \Sc(V(\Ab))$, we can define its partial Fourier transform
\begin{equation}
  \label{eq:partialFT}
 \widehat{\varphi}(\eta) := \int_{\Ab^2} \varphi(\eta_1, r) \psi(\eta_2 \cdot r) dr,
\end{equation}
where $\eta = [\eta_1 \; \eta_2] \in M_2(\Ab) = \mathrm{Hom}(U^\vee, W)(\Ab) = U^\vee \otimes W(\Ab)$.
This defines an intertwining map from $\Sc(V(\Ab))$ to $\Sc(U^\vee\otimes W(\Ab))$ with respect to $\omega$, i.e.
$$
\widehat{\omega(g) \varphi}(\eta) = \widehat{\varphi}(\eta \cdot g)
$$
for $g \in \SL_2(\Ab)$ and $\eta \in M_2(\Ab)$.
There is an equality of theta distribution
\begin{equation}
  \label{eq:theta_dist}
  \sum_{x \in V(\Qb)} \varphi(x) = \sum_{\eta \in M_2(\Qb)} \widehat{\varphi}(\eta) = \sum_{\eta \in  M_2(\Qb)/ \Gamma} \sum_{g \in \Gamma_\eta \backslash \Gamma} \widehat{\varphi}(\eta \cdot g),
\end{equation}
with $\Gamma_\eta \subset \Gamma$ the stabilizer of $\eta \in  M_2(\Qb) / \Gamma$.

Let $g_\tau = n(u) m(\sqrt{v}) \in \SL_2(\Ab)$ be the element corresponding to $\tau = u + iv \in \Hc$. 
It acts on the theta distribution and defines a distribution on $S(V(\Abf))$ 
\begin{equation}
  \label{eq:theta}
  \begin{split}
      \theta(\tau, w, \varphi_\infty)(\varphi_\mathbf{f}) :=
 \sum_{x \in V(\Qb)} \omega(g_\tau) \varphi_\infty(x, w) \varphi_\mathbf{f}(x),
  \end{split}
\end{equation}
where $\varphi_\infty \in S(V(\Rb)) \otimes A^0(\Dc)$.
For a function $f$ on $\Gamma \backslash \Hc$ valued in $\Sc(V(\Abf))$ having mild growth near the cusp, the pairing $\langle f, \theta(\tau, w, \varphi_\infty) \rangle$ is still $\Gamma$-invariant and can be integrated on $\Gamma \backslash \Hc$ to define
\begin{equation}
  \label{eq:Phi}
  \Phi(w; \varphi_\infty, f) := \int_{\Gamma \backslash \Hc} \langle f, \theta(\tau, w, \varphi_\infty) \rangle d \mu(\tau).
\end{equation}
When $f$ has exponential growth at the cusp, the theta integral above can be regularized and is a special case considered in \cite{Borcherds98} (also see \S 1 in \cite{Kudla03}).

\subsection{Cycle Integral of Theta Lift.}
Suppose $M = M_{\af, \kappa}$ and  $w = w(t, t_0)$ as above. Then $\Phi(w; \varphi_\infty, f)$ is also invariant when $\varep \in \Gamma_\kappa$ acts on $t \in \Rb$ via translation by $\log \varep$.
Integrating it with respect to the invariant measure ${dt}$ over a fundamental domain $[0, \log \varepsilon_\kappa)$ of $\Gamma_\kappa \backslash \Dc_F^+$ defines 
\begin{equation}
  \label{eq:I}
  I(t_0; \varphi_\infty, f) := \int_{0}^{\log \varep_\kappa} \Phi(w(t, t_0) ; \varphi_\infty, f) dt.
\end{equation}
Let $\varphi_M \in \Sc(V(\Abf))$ be the characteristic function of $M \otimes \hat{\Zb}$ and
\begin{equation}
  \label{eq:It0}
 \phi(x, w(t, t_0)) := (x, X(t)) (x, Y^\perp(t_0)) e^{-\pi (x, x)_w}.
\end{equation}
In this case, we denote
\begin{equation}
  \label{eq:Theta}
\begin{split}
  \Theta_M(\tau, t, t_0; \phi) &:= v^{-2} \langle \varphi_M, \theta(\tau, w(t, t_0), \phi) \rangle = \sum_{\llam \in M} \varphi_{\tau, w}(\lambda), \\
\varphi_{\tau, w}(\lambda) &:= \omega(g_\tau) \phi(\lambda, w(t, t_0)) = (\llam, X(t)) (\llam, Y^\perp(t_0)) \ebf( Q(\llam_{\wpp}) \tau + Q(\llam_w) \overline{\tau}).
\end{split}
\end{equation}
Notice that $Q(\llam_\wpp ) \tau + Q(\llam_w) \overline{\tau} = Q(\llam) u + Q_w(\llam) iv$ with $\tau = u + iv$. 
Then
\begin{equation}
  \label{eq:It0explicit}
  I(t_0; \phi, \varphi_M) = \int^{\log \varep_\kappa}_0 \int_{\Gamma\backslash \Hc} v^2 \Theta_M(\tau, t, t_0; \phi) d\mu(\tau) dt.
\end{equation}
The first result concerning this function of $t_0$ is as follows.
\begin{prop}
  \label{prop:constant}
  The function $I(t_0; \phi, \varphi_{M_{\af, \kappa}})$ is a constant that only depends on the class of $\af$ in the narrow class group $\Cl^+(F)$ of $F$ and the positive integer $\kappa$.
\end{prop}

\begin{proof}
{  Since $I(t_0; \phi, \varphi_M)$ is differentiable in $t_0$, it suffices to show that $\partial_{t_0} I(t_0) = 0$.
To see this, consider the following elements in $\Sc(V(\Rb))$
\begin{align*}
  \phi_1(x, w(t, t_0)) &:= (x, X(t)) (x, Y(t_0)) e^{-\pi (x, x)_w},  \;
  \phi_2(x, w(t, t_0)) := (x, \Xp(t)) (x, Y(t_0)) e^{-\pi (x, x)_w}.
\end{align*}
Notice that
$$
\partial_{t} X(t) = \Xp(t), \partial_{t} \Xp(t) = X(t), \partial_{t} Y(t) = \Yp(t), \partial_{t} \Yp(t) = Y(t).
$$
From this, it is easy to check that $\partial_{t_0} Q_w(\llam) = (\llam, Y(t_0))(\llam, \Yp(t_0))$.
Therefore, we have
$$
\partial_{t_0} \Theta_M(\tau, t, t_0; \phi) = 
\sum_{\llam \in M}
(1 - 2\pi v (\llam, \Yp(t_0))^2)
 (\llam, X(t)) (\llam, Y(t_0)) \ebf( Q(\llam_{\wpp}) \tau + Q(\llam_w) \overline{\tau}),
$$
which implies
$v^2 \partial_{t_0} \Theta_M(\tau, t, t_0; \phi) d\mu(\tau) dt = 
d \omega$ with
$$
\omega :=
 2 v \Theta_M(\tau, t, t_0; \phi_1) d\tau dt - v^2 \Theta_M(\tau, t, t_0; \phi_2) d\mu(\tau) .
$$
It is clear that $\omega$ is a $(\Gamma \times \Gamma_\kappa)$-invariant 2-form on $\Hc \times \Dc_F^+$.
Applying Stokes' Theorem then proves that $\partial_{t_0} I(t_0) = 0$.}
It is now easy to check from the definition that $\Theta_{M_{\af, \kappa}} = \Theta_{M_{\af(\mu), \kappa}}$ for a totally positive element $\mu \in \Oc_F$. Therefore, $I(t_0; \phi, \varphi_{M_{\af, \kappa}})$ only depends on the class of $\af$ in $\Cl^+(\Oc_F)$ and $\kappa$.
\end{proof}

\begin{rmk}
\label{rmk}
  For $\Ac \in \Cl^+(\Oc_F)$ and $\kappa \in \Nb$, the proposition above implies that the quantity
  \begin{equation}
    \label{eq:invariant}
    \tPsi(\Ac, \kappa) := I(0; \phi, \varphi_{M_{\af, \kappa}}) =
\int^{\log \varepsilon_\kappa}_0 \Phi(t; \af, \kappa) dt 
  \end{equation}
is well-defined with $\af$ any representative of $\Ac$. Here we denote $\Phi(t; \af, \kappa) := \Phi(w(t, 0); \phi, \varphi_{M_{\af,\kappa}})$.
\end{rmk}

From \eqref{eq:iotaw}, we know that when restricting to $t_0 = 0$, the image of $ D_F^+ \times \{0\}$ lies on the twisted diagonal $z_1 = D \kappa z_2 /\Nm(\af)$ in $\Gamma_M^+ \backslash \Hc^2$. On this embedded $\Hc$, the image of $\Gamma_\kappa \backslash \Dc^+_F$ is the geodesic connecting $\sqrt{D} i$ and $\sqrt{D} \frac{-\beta\sqrt{D} + i}{\alpha}$ if $\varep_\kappa = \alpha + \beta\sqrt{D}$. Therefore, $I(0; \phi, \varphi_M)$ can be considered as the cycle integral of the (twisted) diagonal restriction of $\Phi(w; \phi, \varphi_M)$.

The second result of this section relates the Petersson norm of $\vartheta_L$ to the quantity $\tPsi(\Ac, \kappa)$.

\begin{prop}
  \label{prop:norm}
For an integral ideal $\af \subset \Oc_F$ and positive integer $\kappa$, let $L = L_{\af, \kappa}, \varep_\kappa \in \Oc_F^\times$ and $\vartheta_L \in S_1(\Gamma, \rho_L) \otimes \Sc_L$ be as in \S \ref{subsec:Hecke_Theta}. 
Then $\| \vartheta_L(\tau) \|^2_{\mathrm{Pet}} = \tPsi([\af \df] , \kappa) \log \varepsilon_\kappa$.
\end{prop}

\begin{proof}
  If we unfold the definition of $\vartheta_L$ and its Petersson norm in Eqs.\ \eqref{eq:vartheta} and \eqref{eq:PeterssonNorm}, then it is clear that 
$$
\| \vartheta_L(\tau) \|^2_{\mathrm{Pet}} = 
\int^{\log \varepsilon_\kappa}_0\int^{\log \varepsilon_\kappa}_0
\int_{\Gamma \backslash \Hc} v^2 \Theta_{M}(\tau, t, t_0; \phi_2) d\mu(\tau)
dt dt_0,
$$
where $\phi_2$ is defined in the proof of Prop.\ \ref{prop:constant}. 
The proposition now follows from the fact $\Theta_{M_{\af, \kappa}}(\tau, t, t_0; \phi_2) = \Theta_{M_{\af \df, \kappa}}(\tau, t, t_0; \phi)$ and Prop.\ \ref{prop:constant}.
\end{proof}

\subsection{Rademacher symbol.}
\label{subsec:Rademacher}
For each $\gamma = \smat{a}{b}{c}{d} \in \Gamma = \SL_2(\Zb)$, Rademacher defined in \cite{Rademacher56} the following function 
  \begin{equation}
    \label{eq:Psi}
    \Psi(\gamma) = 
    \begin{cases}
\frac{b}{d}, & \text{for } c = 0, \\
\frac{a + d}{c} - 12 \, \sgn(c) \cdot s(a, |c|) - 3\, \sgn(c(a + d)),      &\text{for }c \neq 0,
    \end{cases}
  \end{equation}
where $s(h, k) := \sum_{\mu \bmod k} (\!(\frac{\mu}{k} )\!)  (\!(\frac{\mu h}{k} )\!) $ is the Dedekind sum with $ (\!(x )\!)  :=  x - \lfloor x \rfloor - \half $ if $x \in \Rb \backslash \Zb$ and zero otherwise.
He showed that $\Psi$ is integer-valued, 
invariant under conjugation by elements in $\Gamma$ and 
\begin{equation}
\label{eq:Psieq}
\Psi(\gamma) = \Psi(-\gamma) = -\Psi(\gamma^{-1}) = \det(g) \Psi(g \gamma g^{-1})
\end{equation}
for any $g \in \GL_2(\Zb)$.
We call $\Psi$ the Rademacher symbol. 
It is closely related to the transformation formula of the Dedekind eta function and appears in many places in mathematics (see e.g. \cite{Atiyah87, Za75}).

An element $\gamma = \smat{a}{b}{c}{d} \in \Gamma$ is called \textit{hyperbolic} if $|\tr( \gamma)| > 2$. 
It fixes the semi-circle $C_\gamma := \{ z \in \Hc: |2c z - (a - d)|^2 = (a + d)^2 - 4 \}$.
Define the weight 2 real-analytic Eisenstein series $E^*_2(z)$ by
\begin{equation}
  \label{eq:E2*}
E^*_2(z) :=  - \frac{3}{\pi y} + 1 - 24 \sum_{n \ge 1} \frac{\ebf(nz)}{1-\ebf(nz)}.
\end{equation}
A Theorem of Meyer \cite{Meyer57} connects the cycle integral of $E^*_2(z)dz$ with the Rademacher symbol.
\begin{thm}[see \textsection III of \cite{Za75}]
  Let $\gamma \in \Gamma$ be a hyperbolic element and $z_0 \in C_\gamma$ an arbitrary point.
Then the integral $\int^{\gamma z_0}_{z_0} E^*_2(z) dz$ is independent of $z_0$ and equals to $\Psi(\gamma)$.
\end{thm}

Let $F = \Qb(\sqrt{D})$ and $\varepsilon_\kappa = \alpha + \beta \sqrt{D} \in \Oc_F^\times$ as in \S \ref{subsec:Hecke_Theta}.
Then 
\begin{equation}
\label{eq:gamma}
\gamma = \gamma_{D, \kappa} := \smat{\alpha}{D \beta}{\beta}{\alpha} \in \Gamma
\end{equation}
 is a hyperbolic element, and $C_\gamma = \{z \in \Hc: |z| = \sqrt{D} \}$.
Recall $z(t) := \sqrt{D} \frac{-\sh(t) + i}{\ch(t)}$ from \S \ref{subsec:Witt}, which is an isomorphism from $\Rb$ to $C_\gamma$. It is straightforward to check that
\begin{equation}
\label{eq:translate}
\gamma \cdot z(\log \varepsilon_\kappa) = \gamma \cdot \lp \sqrt{D} \frac{-\beta \sqrt{D} + i}{\alpha} \rp = z(-\log \varepsilon_\kappa).
\end{equation}
Therefore, Meyer's theorem implies that 
\begin{equation}
  \label{eq:cycle1}
  \Psi(\gamma_{D, \kappa}) = -\int^{z(\log \varepsilon_\kappa)}_{z(\log \varepsilon_\kappa^{-1})} E^*_2(z) dz.
\end{equation}
\section{Calculations}
\label{sec:calculations}
\subsection{Partial Fourier Transform.}
Fix a class $\Ac \in \Cl^+(\Oc_F)$, a representative $\af \subset \Oc_F$ of $\Ac$, a positive integer $\kappa \in \Nb$ and denote $M = M_{\af, \kappa}$.
In order to evaluate $\tPsi(\Ac, \kappa)$, we will calculate the Fourier expansion of $\Phi(w; \phi, \varphi_M)$ along the 1-dimensional boundary coming from the stabilizer of $U$. This was done in \cite{Kudla16} with $\varphi_\infty$ being the Gaussian. We follow the same calculations with $\varphi_\infty$ a polynomial times a Gaussian. 
The key step in the evaluation is a 2-dimensional partial Fourier transform. For this, we need the analogue of Lemma 4.3 in \cite{Kudla16} for $\varphi_{\tau, w} \in \Sc(V(\Rb))$ defined in Eq.\ \eqref{eq:Theta}.
The spaces $U(\Rb)$ and $U^\vee(\Rb)$ are both isomorphic to $\Rb^2$ with respect to a choice of basis.
Therefore, we view $\varphi_{\tau, w}$ as a function in $(\eta_1, r)$ with $ \eta_1, r$ column vectors in $\Rb^2$.
The goal is to calculate the partial Fourier transform of $\varphi_{\tau, w}$ defined by 
\begin{equation}
  \label{eq:partial_Fourier}
  \widehat{\varphi_{\tau, w}}(\eta) :=
\int_{\Rb^2} \varphi_{\tau, w}(\eta_1, r) \ebf(\transpose{r} \cdot \eta_2 ) dr,
\end{equation}
where $\eta = [\eta_1 \; \eta_2] \in M_2(\Rb)$.
This is the infinity component of the partial Fourier transform in \eqref{eq:partialFT}.

For the bases of $U$ and $U^\vee$, we will use $\{\tef_1, \tef_2\}$ and $\{\tef^\vee_1, \tef^\vee_2\}$ in \S \ref{subsec:Witt} with $g = g_L$ given in Eq.\ \eqref{eq:gL}.
Denote $J  := \smat{}{1}{-1}{}$, which satisfies $J = - \prescript{t}{}{J}$ and
\begin{equation}
\label{eq:etatau}
\eta_\tau := \eta \binom{\tau}{1} = \eta_1 \tau + \eta_2, \; 
\tilde{\eta}_\tau :=  g^{-1} \cdot \eta_\tau, \;
C := - J \eta_1 /\det(g) . 
\end{equation}
Recall $z(t), z_1(t, t_0), z_2(t, t_0), \tz_1(t), \tz_2(t)$ as in \eqref{eq:tz}, \eqref{eq:iotaw} and \eqref{eq:z}, which satisfy
$$
\tz_1(t, t_0) = g \cdot z_1(t, t_0), \tz_2(t) = \frac{z_2(t, 0)}{\det g}, z_1(t, 0) = \frac{D z_2(t, 0)}{N} = z(t) = \sqrt{D} \frac{-\sh(t) + i}{\ch(t)}.
$$
After setting $t_0 = 0$ and omitting the argument $t$ in all the notations above, we can express $(\llam, w), (\llam, \wpp)$ as
\begin{align*}
(\llam, w) &= \transpose{r} \cdot w_{U'} + \transpose{\eta_1} \cdot w_{U}
= 
\frac{-\zbar}{\sqrt{N} y} (cz + d) 
\lp \transpose{\eta_1} \cdot \binom{-\tz_2}{\tz_1 \tz_2} + \transpose{r} \cdot \binom{\tz_1}{1}\rp \\
& = - \frac{\overline{z_2}}{\sqrt{N}y_2}  \transpose{(r - C z_2)}\cdot g \cdot \binom{z_1}{1}, \\
(\llam, \wpp) &= \frac{\sqrt{N}}{\sqrt{D} y_2} \transpose{(r - C z_2)} \cdot g \cdot  \binom{\overline{z_1}}{1}.
\end{align*}
Then $\varphi_{\tau, w}(\llam)$ becomes
\begin{equation}
\label{eq:varphi_explicit}
\begin{split}
\varphi_{\tau, w}(\eta_1, r) = 
 \frac{-1}{\sqrt{D}y_2^{2}}
& \mathrm{Re}\lp
\overline{z_2} 
\transpose{r'} \cdot R \rp
 \mathrm{Im}\lp  \transpose{r'} \cdot \overline{R} \rp 
 \ebf \lp \tau \prescript{t}{}{r} \cdot \eta_1 \rp 
\mathrm{exp}\lp -\frac{ \pi v}{ y_1 y_2} 
 \left| \transpose{r'} \cdot R \right|^2 \rp,
\end{split}
\end{equation}
where $r' := r - C z_2$ with $R = R(z_1, g) := g \binom{z_1}{1}$.
In the notations above, we have the following lemma.
\begin{lemma}
  \label{lemma:partial}
For $t_0 = 0$ and $\eta \in M_2(\Qb)$ with columns $\eta_1$ and $\eta_2$, the function $\widehat{\varphi_{\tau, w}}(\eta)$ is given by
\begin{equation}
  \label{eq:pFT0}
\begin{split}
  \widehat{\varphi_{\tau, w}}(\eta) = - \frac{N^2 y^2}{\det(g) D^{5/2} v^3}
&\lp
\frac{D v}{2N \pi} - 
\overline{\transpose{\tilde{\eta}_{\tau}}} \cdot 
\pmat{1}{0}{-x}{0}
\cdot \tilde{\eta}_\tau 
\rp\\
&\ebf \lp  \frac{Nz \det(\eta)}{ D \det(g)}  \rp 
\exp \lp - \frac{N \pi}{D v}   \left| \transpose{\tilde{\eta}_\tau} \cdot \binom{1}{-z} \right|^2 \rp.
\end{split}
\end{equation}
\end{lemma}

\begin{prop}
  \label{prop:FT}
The theta function $\Theta(\tau, t, t_0;\phi)$ can be written as
\begin{equation}
  \label{eq:Theta_FT}
  \Theta_M(\tau, t, t_0; \phi) = \sum_{\eta \in M_2(\Qb)} \widehat{\varphi_M}(\eta) \widehat{\varphi_{\tau, w}}(\eta) 
\end{equation}
\end{prop}

This result follows from changing the model of the Weil representation (see \S 4.1 in \cite{Kudla16} for detailed discussions).
Recall the lattice $P \subset M$ defined in \eqref{eq:P} with $M/P \cong L/(L \cap 2 L^\vee)$. 
By (4.10) in \cite{Kudla16} and Example \eqref{example:1} above, we have for $\eta = [\eta_1 \; \eta_2] \in M_2(\Qb) \cong U^\vee \times W$
\begin{equation}
\label{eq:FT_finite}
\widehat{\varphi_M}(\eta) = 
\begin{cases}
\varphi_{M_2(\Zb)}(\eta)  & \mathrm{if } \; L \subset 2 L^\vee, \\
 \varphi_{M_2(\Zb)}(\eta) \ebf \lp -\frac{\det(\eta)}{ 2} \rp & \mathrm{otherwise }.
\end{cases}
\end{equation}

\subsection{Orbital Integrals.}
In this section, we will decompose the integral defining $\Phi(w(t, 0); \phi, \varphi_M)$ into orbits according to the rank of $\eta \in M_2(\Qb)$ and calculate the orbital integral
\begin{equation}
  \label{eq:Ieta}
\Phi_\eta(z, s)  := 
\int_{ \Gamma_\eta \backslash \Hc} v^{2-s} \widehat{\varphi_{\tau, w(t, 0)}}(\eta) \frac{dudv}{v^2} 
\end{equation}
for various $\eta \in M_2(\Qb) / \Gamma$.
Here, we used the notation $z = x + iy = z(t)$ as in Eq.\ \eqref{eq:z}.
The procedure now follows that of \cite{Kudla16}.

\begin{lemma}
  \label{lemma:rank2}
Suppose $g = \smat{a}{b}{0}{d} \in \GL_2^+(\Qb)$.
Let $\eta = \smat{m}{k}{0}{\alpha}\in M_2(\Qb)$ with $m > 0$ and $\alpha \neq 0$. 
Then
\begin{equation}
\label{eq:rank2}
\Phi_\eta(z, 0) 
=
\begin{cases}
-  \frac{N \ebf(m \alpha N z/(D \det(g)))}{\det(g) D} \frac{iyz}{\sqrt{D}},   & \alpha > 0, \\
-  \frac{N \ebf(m \alpha N \zbar/(D \det(g)))}{\det(g) D} \frac{-iy\zbar}{\sqrt{D}},   & \alpha < 0.
\end{cases}
\end{equation}
\end{lemma}

\begin{proof}
Since $\eta$ has rank 2, $\Gamma_\eta$ is trivial and $\Gamma_\eta \backslash \Hc = \Hc$. 
  Substituting in $\eta = \smat{m}{k}{0}{\alpha}$ gives us 
$
\eta_\tau = \binom{m \tau + k}{\alpha}
$
and
\begin{align*}
\widehat{\varphi_{\tau, w}}(\eta) = 
\frac{-N^2 y^2}{\det(g) D^{5/2} v^3} 
&\lp
\frac{D v}{2 N\pi} -
\frac{(d(m \overline{\tau} + k) - \alpha(ax + b))(d(m \tau + k) - \alpha b)}{\det(g)^2}
\rp
\ebf\lp \frac{m \alpha N  z}{D \det{g}}\rp \\
&\exp \lp -\frac{N \pi}{D v \det(g)^2} 
|d(m \tau + k) - \alpha(az + b)|^2
  \rp.
\end{align*}
Make the change of variable $u' = d(mu + k) - \alpha(ax + b)$ and $v' = dmv$ gives us
$$
(d(m \overline{\tau} + k) - \alpha(ax + b))(d(m \tau + k) - \alpha b) =
( u')^2 + (v')^2 + \alpha a x (u' - iv').
$$
Therefore, the integral defining $\Phi_\eta(z, 0)$ becomes
\begin{align*}
\Phi_\eta(z, 0) = 
\frac{-N^2 y^2 \ebf\lp \frac{m \alpha N  z}{D \det{g}}\rp }{\det(g) D^{5/2}} 
\int_0^\infty \int_{-\infty}^\infty
&\lp
\frac{D v'}{2 N\pi } -
\frac{dm((u')^2 + (v')^2 + \alpha a x(u' - iv'))}{\det(g)^2}
\rp\\
&\exp \lp -\frac{N \pi dm ( (u')^2 + (v' - \alpha a y)^2)}{D v' \det(g)^2}   \rp \frac{ du'dv'}{(v')^3}.
\end{align*}
We can apply the identities $\int_\Rb \lp \tfrac{1}{2\pi A} - B u^2 \rp \exp(-\pi A B u^2) du = 0, \int_\Rb \exp(-\pi A u^2) du = 1/\sqrt{A}$ 
and $\int^\infty_0 e^{-A w - B/w} dw /\sqrt{w}  = \sqrt{\pi/A} e^{-2\sqrt{AB}}$
to simplify the expression above to
\begin{align*}
\Phi_\eta(z, s) 
&= \frac{-N^{3/2} y^2 \ebf\lp \frac{m \alpha N  x}{D \det{g}}\rp \sqrt{dm} }{\det(g)^2 D^{2}} 
\int_0^\infty
\lp - \sqrt{v'} + i \alpha a x /\sqrt{v'} \rp
\exp \lp -\frac{N \pi dm (v' + ( \alpha a y)^2/v') }{D \det(g)^2}   \rp \frac{ dv'}{v'}\\
&= 
 \frac{-N^{} i \sgn( \alpha) y  }{\det(g) D^{3/2}} 
( x + i \sgn(\alpha) y) \ebf\lp \frac{m \alpha N  }{D \det(g)} ( x + i \sgn(\alpha) y)\rp.
\end{align*}
This finishes the proof.
\end{proof}

\begin{lemma}
  \label{lemma:rank1}
Suppose $g = \smat{a}{b}{0}{d}$ and $\tz = \tilde{x} + i \tilde{y} := g z = (a z + b)/d$. 
When $\eta = \smat{0}{m}{0}{n} \in M_2(\Qb)$, the orbital integral $\Phi_\eta(z, s)$ is given by 
\begin{equation}
  \label{eq:rank1}
\Phi_\eta(z, s) = 
- \frac{ a \Gamma(1 + s) y^2}{2\sqrt{D} d \pi^2 } \lp \frac{a^2 D}{N \pi} \rp^s
\cdot
\lp
\begin{split}
& \frac{n^2 |\tz|^2 - m^2 + 2(m-n\tilde{x})nb/d}{|m-n\tz|^{4+2s}} - \\
&\frac{  2s(m^2 - mn \tilde{x} + n^2 b \tilde{x}/d - mnb/d)}{|m-n\tz|^{4+2s}}
\end{split}
\rp
\end{equation}
for $s \in \Cb$.
\end{lemma}

\begin{proof}
In this case, we have $\tilde{\eta}_\tau = \frac{1}{ad} \binom{dm - bn}{an}$ and
$$
\widehat{\varphi_{\tau, w}}(\eta) = 
-\frac{N^2 y^2}{D^{5/2} (adv)^3} \lp \frac{Dv(ad)^2}{2N\pi} - d^2 (m - n \tilde{x})(m - nb/d)\rp
\exp \lp -\frac{N\pi}{a^2 D v} |m  - n \tz|^2   \rp.
$$  
Integrating this against $v^{2-s} \frac{dudv}{v^2}$ over $\Gamma_\eta \backslash \Hc = \Gamma_\infty \backslash \Hc$ gives us the desired result.
\end{proof}

\begin{prop}
  \label{prop:rank0}
When $\eta = \smat{0}{0}{0}{0}$, we have $\Gamma_\eta \backslash \Hc = \Gamma \backslash \Hc$ and
\begin{equation}
  \label{eq:rank0}
  \Phi_\eta(z, 0)  = 
- \frac{Ny^2 \mathrm{vol}(\Gamma \backslash \Hc)}{2 \pi \det(g) D^{3/2}} =
- \frac{Ny^2}{6 \det(g) D^{3/2}} .
\end{equation}
\end{prop}

Finally, we will state a lemma that follows readily from differentiating the Kronecker limit formula,
\begin{equation}
  \label{eq:Klf}
\sideset{}{'}\sum_{m, n \in \Zb} \frac{1}{|m - nw|^{2+2s}} = \Im(w)^{-1-s} \lp \frac{\pi}{s} + 2\pi (\gamma - \log 2 - \log(\sqrt{\Im(w)} |\eta(w)|^2)) + O(s) \rp,
\end{equation}
where $\sideset{}{'}\sum$ means summing over $(m, n) \in \Zb^2 \backslash \{(0, 0)\}$, $\gamma$ is the Euler constant and $\eta$ is the Dedekind eta function.
\begin{lemma}
  \label{lemma:kronecker1}
For $w \in \Hc$ and $s$ near $0 \in \Cb$, we have
\begin{align}
  \label{eq:kro1p}
\sideset{}{'}\sum_{m, n \in \Zb} \frac{|w|^2 n^2 - m^2}{|m - nw|^{4 + 2s}}  
&=
\frac{i \pi^2 }{6 \Im(w)} (w E^*_2(w) - \overline{w} E^*_2(-\overline{w})) + O(s),\\
\sideset{}{'}\sum_{m, n \in \Zb} \frac{(m - n \Re(w))n}{|m - nw|^{4 + 2s}}  
&=
- \frac{i \pi^2 }{12 \Im(w)} ( E^*_2(w) -  E^*_2(-\overline{w})) + O(s),\\
2s \sideset{}{'}\sum_{m, n \in \Zb} \frac{(m - n \Re(w))m}{|m - nw|^{4 + 2s}}  
&=
\frac{ \pi }{ \Im(w) }  + O(s),
\end{align}
where $E^*_2(w)$ is the real-analytic Eisenstein series defined in Eq.\ \eqref{eq:E2*}.
\end{lemma}

\section{Proof.}
\label{sec:proof}
After calculating the orbital integrals, we can now add all the contributions together to calculate $\tPsi(\Ac, \kappa)$.
For this, we need the following key proposition.
\begin{prop}
  \label{prop:Phitotal}
Let $M = M_{\af, \kappa}$, $g_0 = g_\kappa, g_1 = g_L$ and $z(t) = \sqrt{D} \frac{-\sh(t) + i}{\ch(t)}$ as in sections \ref{subsec:Witt} and \ref{subsec:lattice}. 
Then 
\begin{equation}
  \label{eq:Phitotal}
      \Phi(t; \af, \kappa) dt =
 \frac{1}{6} \Re \lp 
 E_2^*(\tz_0) d\tz_0 - E_2^*(\tz_1) d\tz_1 \rp,
\end{equation}
where $\tz_j(t) = g_j \cdot z(t)$ for $j = 0, 1$.
\end{prop}

\begin{proof}
First, we can write $\Phi(t; \af, \kappa) = 
\lim_{s \to 0} \int_{\Gamma \backslash \Hc} v^{2-s} \Theta_M(\tau, t, 0; \phi) \frac{dudv}{v^2}$ using the definition of $\Phi(t; \af, \kappa)$ (see Remark \ref{rmk}).
By the absolute convergence of this integral for $\Re(s) > 0$ and Eq.\ \eqref{eq:Theta_FT}, we can write
$$
\int_{\Gamma \backslash \Hc} v^{2-s} \Theta_M(\tau, t, 0; \phi) \frac{dudv}{v^2} = 
\sum_{\eta \in M_2(\Qb) / \Gamma} 
\frac{|\Gamma_\Hc|}{|\Gamma_\Hc \cap \Gamma_\eta|}
\widehat{\varphi_M}(\eta) \Phi_\eta(z, s),
$$
where $\Gamma_\Hc = \{\pm \smat{1}{0}{0}{1} \} \subset \Gamma$.
Suppose $2 \mid D \kappa$ and we choose $g = g_L$ such that $\{ \tef_1, \tef_2\}$ is a $\Zb$-basis of $P_U$ defined in Eq.\ \eqref{eq:P}. Then Eqs.\ \eqref{eq:detgL} and \eqref{eq:FT_finite} tell us that $N/\det(g) = 2/\kappa$ and $\widehat{\varphi_M}(\eta) = \varphi_{M_2(\Zhat)}(\eta)$.
In this case, we have $M_2(\Zb) / \Gamma = S_2 \sqcup S_1 \sqcup \{\smat{0}{0}{0}{0} \}$, where
\begin{align*}
  S_2 &= \left\{
\pmat{m}{k}{0}{\alpha} \in M_2(\Zb): m \in \Zb, m \ge 1, \alpha \in \Zb, \alpha \neq 0, k \bmod m\right\}, \\
S_1 &= \left\{
\pmat{0}{m}{0}{n} \in M_2(\Zb): m, n \in \Zb, (m, n) \neq (0, 0)\right\}.
\end{align*}
It is easy to see that $\frac{|\Gamma_\Hc|}{|\Gamma_\Hc \cap \Gamma_\eta|} = 2$ if $\eta \in S_2$, and is 1 otherwise.
By Eq.\ \eqref{eq:differential} and Lemma \ref{lemma:rank2}, we have
$$
\lim_{s \to 0} \sum_{\eta \in S_2} \Phi_{\eta}(z, s) dt = 
 \sum_{\eta \in S_2} \Phi_{\eta}(z, 0) dt = 
\frac{1}{24} \lp E_2^*(\tz_0) d\tz_0
+ E_2^*(- \overline{\tz_0}) d \overline{\tz_0} 
-
\lp 1 - \frac{3}{\pi \tilde{y}_0} \rp (d {\tz_0}  + d \overline{\tz_0} )
 \rp.
$$
Similarly for the sum over $S_1$, we can apply Lemma \ref{lemma:rank2} and \ref{lemma:kronecker1} to obtain
  \begin{align*}
\lim_{s \to 0} \sum_{\eta \in S_1} \Phi_{\eta}(z, s) dt &= 
- \frac{1}{12} \lp E_2^*(\tz_1) d\tz_1
+ E_2^*(- \overline{\tz_1}) d \overline{\tz_1}
+ \frac{3(d\tz_0 + d \overline{\tz_0})}{\pi \tilde{y}_0}
 \rp.
  \end{align*}
Finally, we can rewrite $\Phi_{\smat{0}{0}{0}{0}}(z, 0) = \frac{d {\tz_0}  + d \overline{\tz_0} }{12}$.
Adding everything together with the index $\frac{|\Gamma_\Hc|}{|\Gamma_\Hc \cap \Gamma_\eta|}$ then finishes the proof for $2 \mid D \kappa$.
If $2 \nmid D \kappa$, then $\widehat{\varphi_M}(\eta) = \varphi_{M_2(\Zb)}(\eta) \ebf(- \det(\eta)/2)$ for $\eta \in M_2(\Qb)$ by Eq.\ \eqref{eq:FT_finite}.
Calculating exactly as before, we obtain Eq.\ \eqref{eq:Phitotal} in this case.
\end{proof}

To prove Theorem \ref{thm:main}, we will also need the following proposition.
\begin{prop}
  \label{prop:tPsi}
Recall $\gamma = \gamma_{D, \kappa} \in \Gamma$ from Eq.\ \eqref{eq:gamma} and let $g_0, g_1 \in \GL^+_2(\Qb)$ be the same as in Prop.\ \ref{prop:Phitotal}.
Denote $\gamma_j := g_j \gamma g_j^{-1}$ for $j = 0, 1$. Then $\gamma_j \in \Gamma$ for $j = 0, 1$ and 
\begin{equation}
  \label{eq:tPsi_explicit}
\tPsi([\af], \kappa) = \frac{1}{12} \lp \Psi(\gamma_1) - \Psi(\gamma_0) \rp,
\end{equation}
where $\tPsi$ is the invariant defined in Eq.\ \eqref{eq:invariant} and $\Psi$ is the Rademacher symbol defined in \S \ref{subsec:Rademacher}.  
\end{prop}

\begin{proof}
  From Prop.\ \ref{prop:Phitotal} and the definition of $\tPsi(\Ac, \kappa)$ in Eq.\ \eqref{eq:invariant}, it is clear that 
$$
\tPsi(\Ac, \kappa) = - \frac{1}{12} \lp \int^{z_0(-\log \varep_\kappa)}_{z_0(\log \varep_\kappa)} \Re(E^*_2(w) dw) -
\int^{z_1(-\log \varep_\kappa)}_{z_1( \log \varep_\kappa)} \Re(E^*_2(w) dw) \rp,
$$
where $\tz_0(t), \tz_1(t)$ were defined in Prop.\ \ref{prop:Phitotal}.
By Eq.\ \eqref{eq:translate}, we know that $\gamma_j \cdot \tz_j(\log \varep_\kappa) = \tz_j(-\log \varep_\kappa)$ and $\gamma_j \cdot \lp - \overline{\tz_j(\log \varep_\kappa)} \rp = - \overline{\tz_j(-\log \varep_\kappa)} $ for $j = 0, 1$.
For any $g = \smat{a}{b}{0}{d} \in \GL_2^+(\Qb)$, it is easy to check that
$$
g \gamma g^{-1} = \frac{1}{a} \pmat{a \alpha + b \beta}{\frac{a^2 D - b^2}{d} \beta}{d \beta}{a \alpha - b \beta}.
$$
One can check that $\gamma_j \in \Gamma$ after substituting in $g = g_j$ for $j = 0, 1$.
Therefore, we will denote $\tz_j = \tz_j(\log \varep_\kappa)$ and can write for $j = 0, 1$
$$
2 \int^{\tz_j(-\log \varep_\kappa)}_{\tz_j( \log \varep_\kappa)} \Re(E^*_2(w) dw) 
= 
 \int^{\gamma_j \tz_j}_{\tz_j} E^*_2(w) dw + E^*_2(-\overline{w}) d \overline{w} 
=
\Psi(\gamma_j) - 
 \int^{- \gamma_j \overline{\tz_j}}_{-\overline{\tz_j}}  E^*_2(w') d w',
$$
where $w' = -\overline{w}$. 
Since $- \gamma_j \overline{\tz_j} = \gamma_j' (- \overline{\tz_j})$ with $\gamma'_j = \smat{1}{}{}{-1} \gamma_j \smat{1}{}{}{-1}$, Meyer's Theorem implies that the last integral is $\Psi(\gamma'_j) = -\Psi(\gamma_j)$ and this finishes the proof.
\end{proof}

\begin{proof}[Proof of Theorem \ref{thm:main}]
  In view of Props.\ \ref{prop:norm}, \ref{prop:Phitotal} and \ref{prop:tPsi}, it suffices to calculate $\tPsi([\af\df], \kappa)$. From the proposition above, we see that $\gamma_0$ only depends on $\kappa$. For $\gamma_1 = \gamma_1(\af\df)$, if $\{a\sqrt{D} + b, d\}$ is a basis of $\af \cap 2 (\kappa \df)^{-1} \af$ as in \S \ref{subsec:lattice}, then $\{a{D} + b\sqrt{D}, d\sqrt{D}\}$ is a basis of $\df \af \cap 2 \kappa^{-1} \af$.
Let $r = \gcd(b, d)$, $b' = b/r, d' = d/r$ and $g_2 = \smat{a'}{c'}{d'}{-b'} \in \GL_2(\Zb)$ with $\det(g_2)  = -1$.
Then $\{r \sqrt{D} + a a'D, a d'D\}$ is a basis of $\af \df \cap 2 \kappa^{-1} \af$ and we choose $g_1 = \smat{r}{aa'D}{}{ad'D}$. 
It is straightforward to check that 
$$
g_1
 = g_2
\pmat{a}{b}{}{d} \pmat{}{D}{1}{}
$$
Since $\smat{}{D}{1}{} \gamma_{} \smat{}{1}{1/D}{} =  \gamma$, we have $\Psi(\gamma_1(\af\df)) = \Psi(g_2 \gamma_1(\af) g_2^{-1}) = - \Psi(\gamma_1(\af))$. This finishes the proof.
\end{proof}

\bibliography{norm_bib.bib}{}
\bibliographystyle{amsplain}
\end{document}